\documentclass[11pt]{article}
\usepackage{amsmath}
\usepackage{amsthm}
\usepackage{amssymb}
\usepackage{mathrsfs}
\usepackage{xcolor}

\usepackage{geometry}
\usepackage[colorlinks=true,
linkcolor=blue,citecolor=blue,
urlcolor=blue]{hyperref}

\makeatletter
\def\@seccntDot{.}
\def\@seccntformat#1{\csname the#1\endcsname\@seccntDot\hskip 0.5em}
\renewcommand\section{\@startsection{section}{1}{\z@}%
{18\p@ \@plus 6\p@ \@minus 3\p@}%
{9\p@ \@plus 6\p@ \@minus 3\p@}%
{\large\bfseries\boldmath}}
\renewcommand\subsection{\@startsection{subsection}{2}{\z@}%
{12\p@ \@plus 6\p@ \@minus 3\p@}%
{3\p@ \@plus 6\p@ \@minus 3\p@}%
{\it}}
\renewcommand\subsubsection{\@startsection{subsubsection}{3}{\z@}%
{12\p@ \@plus 6\p@ \@minus 3\p@}%
{\p@}%
{\bfseries\boldmath}}
\makeatother

\usepackage{microtype}

\theoremstyle{plain}
\newtheorem{theorem}{Theorem}[section]
\newtheorem{lemma}{Lemma}[section]

\theoremstyle{definition}

\newtheorem{remark}{Remark}[section]

\newtheorem{claim}{Claim}[section]

\numberwithin{equation}{section}
\allowdisplaybreaks
\newcommand{\x}{{\bf x}}

\title{Maximum principal ratio of the signless Laplacian of graphs}

\author{
Lele Liu\thanks{College of Science, University of Shanghai for Science and Technology, Shanghai 200093, China
(\texttt{ahhylau@outlook.com}). This author is supported by the National Natural Science Foundation of China (No. 12001370).}
\and
Shengming Hu\thanks{College of Science, University of Shanghai for Science and Technology, Shanghai 200093, China
(\texttt{hsm606125@163.com})}
\and 
Changxiang He\thanks{Corresponding author. College of Science, University of Shanghai for Science and Technology, Shanghai 200093, China
(\texttt{changxiang-he@163.com})}
}

\date{}

\begin{document}
\maketitle

\begin{abstract}
Let $G$ be a connected graph and $Q(G)$ be the signless Laplacian of $G$. The principal ratio $\gamma(G)$
of $Q(G)$ is the ratio of the maximum and minimum entries of the Perron vector of $Q(G)$. In this paper, 
we consider the maximum principal ratio $\gamma(G)$ among all connected graphs of order $n$, and show 
that for sufficiently large $n$ the extremal graph is a kite graph obtained by identifying an end vertex 
of a path to any vertex of a complete graph.
\par\vspace{2mm}

\par\vspace{1mm}

\end{abstract}

\section{Introduction}

In this paper, we consider only simple, undirected graphs, i.e, undirected graphs without multiple edges or loops.
The signless Laplacian of a graph $G$ is defined as $Q(G)=D(G)+A(G)$, where $D(G)$ is the diagonal matrix of vertex 
degrees of $G$, and $A(G)$ is the adjacency matrix of $G$. The largest eigenvalue of $Q(G)$, denoted by $q_1(G)$, 
is referred to as the \emph{$Q$-spectral radius} of $G$. For a connected graph $G$, the Perron--Frobenius theorem 
implies that $Q(G)$ has a unique positive unit eigenvector $\x$ corresponding to $q_1(G)$, which is called the 
\emph{principal $Q$-eigenvector} of $G$. Let $x_{\min}$ and $x_{\max}$ be the smallest and the largest entries 
of $\x$, respectively. We define the principal ratio $\gamma(G)$ of $Q(G)$ as
\[
\gamma(G):= \frac{x_{\max}}{x_{\min}}.
\]
Evidently, $\gamma(G)\geq 1$ with equality if and only if $G$ is regular. Therefore, it can be considered as a measure 
of graph irregularity. In this paper, we consider the principal ratio of graphs and determine the unique extremal 
graph maximizing $\gamma(G)$ among all connected graphs on $n$ vertices for sufficiently large $n$.

The principal ratio of the adjacency matrices of graphs has been well studied. In 1958, Schneider 
\cite{Schneider1958} presented an upper bound on eigenvectors of irreducible nonnegative matrices; for graphs it 
can be described as $\gamma(G)\leq (\lambda_1(G))^{n-1}$, where $\lambda_1(G)$ is the largest eigenvalue of the 
adjacency matrix of $G$. Nikiforov \cite{Nikiforov2007} improved this result for estimating the gap of spectral 
radius between $G$ and its proper subgraph. Subsequently, Cioab\u a and Gregory \cite{CioabGregory2007} slightly 
improved the previous results for $\lambda_1(G) > 2$. In addition, they also proved some lower bounds on $\gamma(G)$, 
which improved previous results of Ostrowski \cite{Ostrowski1952} and Zhang \cite{Zhang2005}.

Recall that the {\em kite} or {\em lollipop} graph, denoted $P_r \cdot K_s$, is obtained by identifying an end 
vertex of the path $P_r$ to any vertex of the complete graph $K_s$. In 2007, Cioab\u a and Gregory \cite{CioabGregory2007} 
initially posed the conjecture that among all connected graphs of order $n$, the kite graph attains the maximum 
ratio of the largest and smallest Perron vector entries of $A(G)$. In 2018, Tait and 
Tobin \cite{TaitTobin2018} confirmed the conjecture for sufficiently large $n$. Recently, Liu and He \cite{LuHe2021} 
improved the condition as $n\geq 5000$. Using the method posed by Tait and 
Tobin \cite{TaitTobin2018} we can prove the main result of this paper as follows.

\begin{theorem}\label{thm:Main}
For sufficiently large $n$, the connected graph $G$ on $n$ vertices with maximum principal ratio of $Q(G)$ is a kite graph.
\end{theorem}

\section{Preliminaries}

For graph notation and concepts undefined here, we refer the reader to \cite{BondyMurty2008}. Given a connected 
graph $G$ on $n$ vertices and a vertex $v$ in $G$, we write $N_G(v)$ for the set of neighbors of $v$ and $d_G(v)$ 
the degree of $v$ in $G$. For convenience, we drop the subscript when it is understood. Recall that a \emph{pendant path} 
is a path with one end vertex of degree one and all the internal vertices of degree two. 

Let $G$ be a graph with vertex set $V(G)=\{v_1,v_2,\ldots,v_n\}$. If $q_1(G) > 4$, we denote
\[
\sigma(G):= \frac{q_1(G) - 2 + \sqrt{q_1(G)^2 - 4q_1(G)}}{2}.
\]
Let $\x$ be the principal $Q$-eigenvector of $G$. Hereafter, we write $x_i$ for the entry of $\x$ corresponding 
to the vertex $v_i\in V(G)$. 

The following lemma give an upper bound for $\gamma(G)$ in terms of $\sigma(G)$.

\begin{lemma}\label{lem:gamma}
Let $G$ be a connected graph, and $\x$ be the principal $Q$-eigenvector of $G$. Suppose that 
$v_1,v_2,\ldots,v_k$ is a shortest path between $v_1$ and $v_k$, where $v_1$ attains the minimum component of 
$\x$ and $v_k$ attains the maximum. If $q_1(G) > 4$ and $x_k=1$, then for $1 \leq j \leq k$,
\begin{equation}\label{eq:gamma}
\gamma(G)\leq\frac{\sigma(G)^j - \sigma(G)^{-j} + \sigma(G)^{j-1} - \sigma(G)^{-(j-1)}}{\sigma(G) - \sigma(G)^{-1}} 
\cdot\frac{1}{x_j}, 
\end{equation}
with equality if $v_1,v_2,\ldots,v_j$ form a pendant path.
\end{lemma}

\begin{proof}
Set $q:=q_1(G)$ for short. By the eigenvalue equations of $Q(G)$ we have
\begin{align*}
(q - 1) x_1 & \geq x_2 \\
(q - 2) x_2 & \geq x_1 + x_3 \\
(q - 2) x_3 & \geq x_2 + x_4 \\
& ~\,\vdots \\
(q - 2) x_{j-1} & \geq x_{j-2} + x_j.
\end{align*}
Based on the first two inequalities, we have
\[
x_1 \geq \frac{x_2}{q - 1},~~ x_2 \geq \frac{q - 1}{(q - 2)(q - 1) - 1} x_3.
\]
Now we assume that 
\[
x_i \geq \frac{U_{i-1}}{U_i} x_{i+1}
\]
with $U_j$ positive for all $j \leq i$. Since $(q - 2) x_{i+1} \geq x_i + x_{i+2}$, we get 
\[
x_{i+1}\geq\frac{U_i}{(q - 2) U_i - U_{i-1}} x_{i+2},
\]
where $(q - 2) U_i - U_{i-1} > 0$ as $(q - 2) x_{i+1} > x_i \geq x_{i+1} U_{i-1}/U_i$. 
So we let 
\[
U_{i+1} = (q - 2) U_i - U_{i-1}
\]
with $U_0=1$, $U_1 = q-1$. Solving this recurrence and using the initial conditions, we obtain 
\begin{equation}\label{eq:u_i}
U_i = \frac{\sigma(G)^{i+1} - \sigma(G)^{-(i+1)} + \sigma(G)^i - \sigma(G)^{-i}}{\sigma(G) - \sigma(G)^{-1}}.
\end{equation}
As a consequence,
\[
x_1 \geq \frac{U_0}{U_1} \cdot x_2 \geq \prod_{i=1}^2 \frac{U_{i-1}}{U_i} \cdot x_3
\geq \cdots \geq \prod_{i=1}^{j-1} \frac{U_{i-1}}{U_i} \cdot x_j
= \frac{x_j}{U_{j-1}}.
\]
Hence, for $1 \leq j \leq k$,
\[
\gamma(G)=\frac{x_k}{x_1} = \frac{1}{x_1}\leq \frac{U_{j-1}}{x_j}.
\]

Finally, if $v_1,v_2,\ldots,v_j$ form a pendant path, then we have all equalities 
throughout, as desired.
\end{proof}

\autoref{lem:gamma} will be used frequently in the sequel. For the sake of convenience, we denote
\[
U_i(G):= \frac{\sigma(G)^{i+1} - \sigma(G)^{-(i+1)} + \sigma(G)^i - \sigma(G)^{-i}}{\sigma(G) - \sigma(G)^{-1}}
\] 
for a connected graph $G$.

\begin{remark}
Consider the kite graph $P_k \cdot K_{n-k+1}$. It is straightforward to check that the smallest entry 
of the principal $Q$-eigenvector is the vertex of degree $1$ and the largest is the vertex of degree 
$n - k + 1$.  By \autoref{lem:gamma}, we have 
\[
\gamma(P_k \cdot K_{n-k+1}) = U_{k-1}(P_k \cdot K_{n-k+1}).
\]
\end{remark}

\begin{lemma}\label{lem:bound-Ui}
Let $j\geq 2$ and $q:= q_1(G) > 4$. Then
\begin{equation}\label{eq:u}
(q - 1)\Big(q - 2 - \frac{1}{q - 3}\Big)^{j-2} \leq U_{j-1}(G)
\leq (q - 1)\Big(q - 2 - \frac{1}{q}\Big)^{j-2}.  
\end{equation}
\end{lemma}

\begin{proof}
Set $U_i:= U_i(G)$ for short. Recall that $U_1 = q - 1$ and $U_{i+1} = (q-2) U_i - U_{i-1}$. 
If $j=2$, the assertion holds trivially, so we assume that $j\geq 3$. 

Observe that
\[
\frac{U_{j-1}}{q - 1}=\frac{U_{j-1}}{U_1} = \prod_{i=2}^{j-1} \frac{U_i}{U_{i-1}},
\]
it turns to prove that for $2 \leq i \leq j-1$,
\[
q - 2 - \frac{1}{q - 3}\leq\frac{U_i}{U_{i-1}}\leq q - 2 - \frac{1}{q}.
\]
We first prove the right-hand side of \eqref{eq:u} by induction on $i$. For $i=2$, 
recall that $U_0=1$ and $U_1=q-1$ we have
\[
\frac{U_2}{U_1}=\frac{(q - 2) U_1 - U_0}{U_1} = q - 2 - \frac{1}{q - 1} < q - 2 - \frac{1}{q}.
\]
By induction, we assume that $U_{j-2}/U_{j-3} < q - 2 - 1/q$. Then
\[
\frac{U_{j-1}}{U_{j-2}} = \frac{(q - 2) U_{j-2} - U_{j-3}}{U_{j-2}}
< q - 2 - \frac{1}{q - 2 - 1/q} < q - 2 - \frac{1}{q},
\]
as desired. Likewise, the inequality of the left-hand side of \eqref{eq:u} can be proved by 
analogous arguments as above.
\end{proof}

\section{Proof of \autoref{thm:Main}}

Throughout the remainder of this paper, we always assume $G$ is a graph maximizing $\gamma(G)$ 
among all connected graphs on $n$ vertices, where $n$ is large enough. Let $V(G)=\{v_1,v_2,\ldots,v_n\}$
and $\x$ be the principal $Q$-eigenvector of $G$ with maximum entry is $1$. Without loss of generality, 
assume that $v_1$ attains the minimum entry of $\x$, while $v_k$ attains the maximum, that is, $x_k=1$.
Suppose that $v_1,v_2,\ldots,v_k$ is a shortest path of length $(k-1)$ between $v_1$ and $v_k$. 
Hereafter, we use $C$ to denote the set $V(G)\setminus\{v_1,\ldots,v_k\}$, and write $S:= C\cap N(v_{k-1})$. 
For convenience, we always set $q:=q(G)$, $\sigma:=\sigma(G)$ and $U_j := U_j(G)$. 

In this section, the proof of our main result is presented. To this end, we divide this section 
into three subsections. First, we show that the vertices $v_1,v_2,\ldots,v_{k-2}$ form a pendant 
path and that $v_k$ is connected to all of the vertices that are not on this path 
(Subsection \ref{subsection:auxiliary-results}). Next, we show that $v_{k-2}$ has degree exactly 
two (Subsection \ref{subsection:degree-v-k2-2}). Based on previous results, we then show that 
$v_{k-1}$ also has degree exactly two (Subsection \ref{subsection:degree-v-k1-2}). 
Finally, we prove that connecting any non-edge in $V(G)\setminus \{v_1,\ldots,v_k\}$ will increase 
the principal ratio, and hence, the extremal graph is exactly a kite graph.

Let us remark that $q > 4$. Indeed, if $q\leq 4$, then $G$ must be one of the path $P_n$, the cycle 
$C_n$ and the star $K_{1,3}$ whose principal ratio is less than that of $P_{n-2}\cdot K_3$.

\subsection{Some auxiliary results}
\label{subsection:auxiliary-results}

\begin{lemma}\label{lem:auxiliary-results}
The following statements hold.
\begin{enumerate}
\item[$(1)$] $d(v_k) = n-k+1$.
\item[$(2)$] $2(n-k)< q <2(n-k+1)$.
\item[$(3)$] $v_1,v_2,\ldots,v_{k-2}$ form a pendant path in $G$.
\end{enumerate} 
\end{lemma}

\begin{proof}
Let $H=P_k \cdot K_{n-k+1}$. By maximality and \autoref{lem:gamma}, we see
\[
U_{k-1}(G)\geq\gamma(G)\geq\gamma(H)=U_{k-1}(H).
\]
Notice that the function 
\[
f(x):= \frac{x^k - x^{-k} + x^{k-1} - x^{-(k-1)}}{x - x^{-1}}
\]
is increasing whenever $x\geq 1$. Hence, $\sigma(G)\geq\sigma(H)$. It follows that $q\geq q_1(H) > 2(n-k)$.

According to eigenvalue equation for $v_k$ and $x_k=1$, we find that
\begin{equation}\label{eq:eig-eq-for-q-upper-bound}
2(n-k) < q = q x_k = d(v_k) x_k + \sum_{v\in N(v_k)} x_v < 2 d(v_k),
\end{equation}
which implies $d(v_k)\geq n-k+1$. On the other hand, $v_k$ has no neighbors in $\{v_1,\ldots,v_{k-2}\}$,
as otherwise there would be a shorter path between $v_1$ and $v_k$. Thus, $d(v_k)\leq n-k+1$.
So we obtain $d(v_k) = n-k+1$. The inequality of the right-hand side of item (2) follows from 
\eqref{eq:eig-eq-for-q-upper-bound} and $d(v_k) = n-k+1$.

Finally, we prove the item (3). Since $d(v_k)=n-k+1$, we have $N(v_k) = C\cup\{v_{k-1}\}$. It 
follows that $v_1,\ldots,v_{k-3}$ have no neighbors off the path, otherwise there would be a 
shorter path between $v_1$ and $v_k$. Hence, $v_1,v_2,\ldots,v_{k-2}$ form a pendant path in $G$.
\end{proof}

To prove our main result, we need to make an estimation on $k$, as stated in the following lemma.

\begin{lemma}\label{lem:n-k-value}
$n-k = (1 + o(1)) \frac{n}{\log n}$.
\end{lemma}

\begin{proof}
Let $H=P_j \cdot K_{n-j+1}$, where $j = \lfloor n - n/\log n\rfloor$. From \autoref{lem:gamma} 
and \autoref{lem:bound-Ui}, we find that 
\[
\gamma(H) = U_{j-1}(H) > (q_1(H)-3)^{j-1},
\]
and
\[
\gamma(G)\leq U_{k-1} < (q-1)^{k-1}.
\]
Since $q_1(H) > 2(n-j)$ and $q < 2(n-k+1)$, by the maximality of $\gamma(G)$, we have 
\begin{equation}\label{eq:ineq-for-solv-k}
\big(2(n-k) + 1\big)^{k-1} > \big(2(n-j) - 3\big)^{j-1}.
\end{equation}
Solving \eqref{eq:ineq-for-solv-k}, we obtain $n - k = (1+o(1)) \frac{n}{\log n}$.
\end{proof}

\begin{lemma}\label{lem:x-norm}
$q/2 - 2 < \|\x\|^2 < q/2 + 3$.
\end{lemma}

\begin{proof}
We first prove that $\|\x\|^2 > q/2 - 2$. Indeed,
\[
\|\x\|^2 > x_{k}^{2} + \sum_{v\in N(v_k)} x_v^{2} 
> 1 + \frac{\big(\sum_{v\in N(v_k)} x_v\big)^2}{n-k+1} 
= 1 + \frac{\big(q-(n-k+1)\big)^2}{n-k+1}.
\]
Since $q > 2(n-k)$, we see $n-k+1 < q/2 + 1$. It follows that
\[
\|\x\|^2 > 1 + \frac{(q/2 - 1)^2}{q/2 + 1} > \frac{q}{2} - 2.
\]

Next, we shall prove the right-hand side. Set $x_0=0$. Then for $1\leq i \leq k-3$, we have
\[
(q - 2) x_i = x_{i-1} + x_{i+1} < x_{i} + x_{i+1},
\]
which implies that
\[ 
x_i < \frac{x_{i+1}}{q-3} < \cdots < \frac{x_{k-2}}{(q-3)^{k-2-i}} < \frac{1}{(q-3)^{k-2-i}}.
\] 
It follows from the above inequalities and $q < 2(n-k+1)$ that
\[
\|\x\|^2 < x_k + x_{k-2} + \sum_{v\in N(v_k)} x_v + \sum_{i=1}^{k-3} x_i < \frac{q}{2} + 3,
\]
as desired.
\end{proof}

\begin{lemma}\label{lem:U-bound}
For every subset $U$ of $N(v_k)$, we have
\begin{equation}\label{eq:U}
|U| - 2 < \sum_{v\in U} x_v \leq |U|.
\end{equation}
\end{lemma}

\begin{proof}
The upper bound is clear from $x_v\leq 1$ for $v\in V(G)$. The lower bound follows from the 
inequalities
\[
\sum_{v\in N(v_k)\setminus U} x_v \leq |N(v_k)| - |U|,
\] 
and
\[
\sum_{v\in N(v_k)} x_v = q - |N(v_k)| > |N(v_k)| - 2.
\]
The last inequality is due to the fact $q > 2(n-k)$. This completes the proof of the lemma.
\end{proof}

\subsection{ The vertex degree of $v_{k-2}$ is two}
\label{subsection:degree-v-k2-2}

\begin{lemma}\label{lem:degree-k-2-equal-2}
The vertex $v_{k-2}$ has degree exactly $2$ in $G$.
\end{lemma}

\begin{proof}
Assume by contradiction that $d(v_{k-2})\geq 3$. Set $T:= N(v_{k-2})\cap N(v_k)$ for short. 
Then $d(v_{k-2}) = |T| + 1$ and $|T|\geq 2$. Our proof hinges on the following claims.

\begin{claim}\label{claim:sum-x-leq-1}
$\sum_{v\in T} x_v < 1$.
\end{claim}

\begin{proof}
By the eigenvalue equation for $v_{k-2}$, we get
\[ 
\big(q - (|T|+1)\big) x_{k-2} = x_{k-3} + \sum_{v\in T} x_v.
\]
Noting that $x_{k-3} = x_{k-2} U_{k-4}/U_{k-3}$ and $U_{k-2} = (q - 2) U_{k-3} - U_{k-4}$, we have 
\[ 
\bigg(\frac{U_{k-2}}{U_{k-3}} - (|T| - 1) \bigg) x_{k-2} = \sum_{v\in T} x_v.
\]
In light of \autoref{lem:gamma} and \autoref{lem:auxiliary-results}, we have 
$\gamma(G) = U_{k-3}/ x_{k-2}$. Therefore,
\begin{equation}\label{eq:gamma-in-degree-k-2}
\gamma(G) = \frac{U_{k-3}}{x_{k-2}} = \big(U_{k-2} - (|T| - 1) U_{k-3}\big)
\bigg(\sum_{v\in T} x_v\bigg)^{-1}.
\end{equation}
If $\sum_{v\in T} x_v\geq 1$, it follows from \eqref{eq:gamma-in-degree-k-2} that $\gamma(G) < U_{k-2}$. 
On the other hand, let $H=P_{k-1} \cdot K_{n-k+2}$. By the maximality and \autoref{lem:gamma}, we get
\[
\gamma(G) \geq \gamma(H) = U_{k-2}(H),
\]
which, together with $\gamma(G) < U_{k-2}$, implies that $q > q(H) > 2(n-k+1)$, a contradiction 
yielding $\sum_{v\in T} x_v<1$.
\end{proof}

\begin{claim}\label{claim:degree-k-2-geq-3}
$d(v_{k-2})\leq 3$. 
\end{claim}

\begin{proof}
If $d(v_{k-2})\geq 4$, then $|T|\geq 3$. From \autoref{lem:U-bound}, we have 
$\sum_{v\in T} x_v > |T| - 2\geq 1$, a contradiction to \autoref{claim:sum-x-leq-1}.
\end{proof}

Combining with our assumption for contradiction that $d(v_{k-2})\geq 3$, we derive that 
$d(v_{k-2}) = 3$. Let $u$ be the unique vertex in $C$ adjacent to $v_{k-2}$. 
We may assume that $x_{k-1}\leq x_u$, otherwise we choose another path $v_1,\ldots,v_{k-2},u,v_k$.

Now we continue to prove this lemma by considering the vertex degree of $v_{k-1}$. 
For simplicity, write $d := d(v_{k-1})$. Recall that $S=N(v_{k-1})\cap C$. Then $d = |S| + 2$. 
\par\vspace{1.5mm}

\noindent {\bfseries Case 1.} $d(v_{k-1}) \geq 10$.
Let $G_1^{-} = G - \{v_{k-2}u\}$. Then $v_1,v_2,\ldots,v_{k-1}$ form a pendant path in $G_1^-$.
Set $q^{-}_1 := q_1(G_1^{-})$. By Rayleigh principle and \autoref{lem:x-norm},
\begin{equation}\label{eq:rayleigh-principle}
q_1^{-} - q \geq \frac{\x\big(Q(G^{-}_1) - Q(G)\big)\x}{\|\x\|^2} \geq
- \frac{4}{\|\x\|^2}\geq -\frac{8}{q - 4}.
\end{equation}

Let $\x^{-}$ be the principal $Q$-eigenvector of $G_1^{-}$ with maximum entry $1$. To compare 
$\gamma(G_1^-)$ with $\gamma(G)$, we use \autoref{lem:gamma} and \autoref{lem:bound-Ui} to bound 
them. Then we have
\[
\gamma(G_1^{-}) = \frac{U_{k-2}(G_1^{-})}{x_{k-1}^-} 
> \Big(q^{-}_1 - 2 - \frac{1}{q^{-}_1 - 3}\Big)^{k-3} \cdot \frac{q^{-}_1 - 1}{x_{k-1}^-},
\]
and
\begin{equation}\label{eq:gamma-G}
\gamma(G) = \frac{U_{k-3}}{x_{k-2}} < \Big(q - 2 - \frac{1}{q}\Big)^{k-4} \cdot \frac{q - 1}{x_{k-2}}.  
\end{equation}
Combining the above two inequalities, we deduce that
\begin{equation}\label{eq:G3}
\frac{\gamma(G_1^{-})}{\gamma(G)} > \frac{q^{-}_1 - 1}{q - 1}\cdot
\Big(\frac{q^{-}_1 - 2 - (q^{-}_1 - 3)^{-1}}{q - 2 - 1/q}\Big)^{k-4}
\big(q^{-}_1 - 3\big) \cdot \frac{x_{k-2}}{x_{k-1}^-}.   
\end{equation}
To proceed further, we first consider the ratio $x_{k-2}/x_{k-1}^-$. Since 
$(q-3) x_{k-2} = x_{k-3} + x_{k-1} + x_u > 2 x_{k-1}$, we see 
\[
\frac{x_{k-2}}{x_{k-1}}>\frac{2}{q-3}.
\]
Furthermore, by \autoref{lem:U-bound}, 
\[ 
(q-d)x_{k-1} > x_k + \sum_{v\in S} x_v > |S| - 1 = d-3.
\] 
Then $x_{k-1} > (d - 3)/(q - d)$. In addition, $x_{k-1}^- < d/(q^{-}_1 - d)$. In view of 
\eqref{eq:rayleigh-principle} and $d\geq 10$, we deduce that
\[
\frac{x_{k-1}}{x_{k-1}^-} > \Big(1 - \frac{3}{d}\Big) \Big(\frac{q^{-}_1 - d}{q - d}\Big) > \frac{3}{5}.
\]
As a consequence, 
\[
\frac{x_{k-2}}{x_{k-1}^-} = \frac{x_{k-2}}{x_{k-1}} \cdot \frac{x_{k-1}}{x_{k-1}^-} 
> \frac{6}{5 (q-3)}.
\]
Combining with \eqref{eq:G3} and \eqref{eq:rayleigh-principle}, we obtain 
\begin{align*}
\frac{\gamma(G_1^{-})}{\gamma(G)} 
& > \frac{11}{10} \cdot \Big(\frac{q^{-}_1 - 2 - (q^{-}_1 - 3)^{-1}}{q - 2 - 1/q}\Big)^{k-4} \\
& > \frac{11}{10} \cdot \Big(1 - \frac{10}{q^2}\Big)^{k-4} \\
& > \frac{11}{10} \cdot \Big(1 - \frac{10(k-4)}{q^2}\Big) \\
& > 1.
\end{align*}
The third inequality follows from Bernoulli's inequality. Hence, $\gamma(G_1^-) > \gamma(G)$, a contradiction.
\par\vspace{1.5mm}

\noindent {\bfseries Case 2.} $d(v_{k-1}) \leq 9$.
Let $G_2^{-}$ be the graph obtain from $G$ by deleting edges $\{v_{k-1}v: v\in S\}$ and 
$\{uv_{k-2}\}$. Denote $q^{-}_2:= q_1(G_2^{-})$. 
We have
\[
\gamma(G_2^{-}) = U_{k-1}(G_2^{-}) > (q^{-}_2 - 1) \Big(q^{-}_2 - 2 - \frac{1}{q^{-}_2 - 3}\Big)^{k-2}.
\]
It follows from \eqref{eq:gamma-G} that
\[
\frac{\gamma(G_2^{-})}{\gamma(G)} > \frac{q^{-}_2 - 1}{q - 1} \cdot 
\Big(\frac{q^{-}_2 - 2 - (q^{-}_2 - 3)^{-1}}{q - 2 - 1/q}\Big)^{k-4}
(q^{-}_2 - 3)^2 \cdot x_{k-2}.
\]
Since $(q-3) x_{k-2} > 2 x_{k-1}$ and $(q-2) x_{k-1}\geq (q-d(v_{k-1})) x_{k-1} > 1$, 
then 
\[
x_{k-2} > \frac{2}{(q-2)(q-3)}.
\]
Using similar arguments as above, we have $\gamma(G_2^{-}) > \gamma(G)$, a contradiction.
\end{proof}

\subsection{The vertex degree of $v_{k-1}$ is two}
\label{subsection:degree-v-k1-2}

The next lemma gives a more precise upper bound for $q_1(G)$.

\begin{lemma}\label{lem:small-q}
$q < 2(n-k) + 3/2$.
\end{lemma}

\begin{proof}
Assume for contradiction that $q\geq 2(n-k) + 3/2$. We have the following claims.

\begin{claim}\label{claim:xu-large}
$x_v\geq 1/2$ for each $v\in N(v_k)$.
\end{claim} 

\begin{proof} 
If there is a vertex $w\in N(v_k)$ such that $x_w< 1/2$, then 
\[
n-k + \frac{1}{2} \leq (q - (n-k+1)) x_k = \sum_{u\in N(v_k)} x_u < \frac{1}{2} + (n-k),
\]
a contradiction completing the proof of the claim.
\end{proof}

\begin{claim}\label{claim}
There is a vertex $z\in C$ such that $v_{k-1}z\notin E(G)$. 
\end{claim}

\begin{proof}
Assume for contradiction that $v_{k-1}v\in E(G)$ for each $v\in C$. From the eigenvalue equations for 
$v_{k-1}$ and $v_k$, we get $(q-n+k-1) (x_{k-1} - x_k) = x_k + x_{k-2} > 0$, and therefore 
$x_{k-1} > x_k = 1$, which leads to a contradiction. This completes the proof of the claim.
\end{proof}

Now, let $G^{+} = G + \{v_{k-1}z\}$, and $\x^+$ be the principal $Q$-eigenvector of $G^+$ with maximum 
entry $1$. Our goal is to show $\gamma(G^+)>\gamma(G)$, and therefore deduce a contradiction.

Set $q^+:=q_1(G^+)$ for short. By the Rayleigh principle and \autoref{claim:xu-large} we obtain
\[
q^+ - q \geq \frac{\x^{\mathrm{T}} Q(G^+)\x-\x^{\mathrm{T}} Q(G)\x}{\|\x\|^2}
=\frac{(x_{k-1} + x_z)^2}{\|\x\|^2}>\frac{1}{\|\x\|^2},
\]
which, together with \autoref{lem:x-norm}, gives
\begin{equation}\label{eq:lambda-plus-lambda}
q^+ > q + \frac{2}{q+6}.
\end{equation}

The next claim gives an upper bound for $x_{k-1}^+ - x_{k-1}$.

\begin{claim}\label{claim:xplus-x}
$x_{k-1}^+ - x_{k-1} < 4/(n-k)$.
\end{claim}

\begin{proof} 
By the eigenvalue equations for $q$ and $q^+$ with respect to $v_{k-1}$, we have
\begin{align*}
(q - |S| - 2) x_{k-1} & = x_{k-2}+x_k+\sum_{v\in S} x_v, \\
(q^+ - |S| - 3) x_{k-1}^+ & = x_{k-2}^+ + x_k^+ + x_z^+ +\sum_{v\in S} x_v^+.
\end{align*}
It follows from $q^+ > q$ that
\begin{align*}
(q - |S| - 2) (x_{k-1}^+ - x_{k-1}) 
& < (x_{k-2}^+ - x_{k-2}) + (x_k^+ + x_{k-1}^+ + x_z^+ - 1)
+ \sum_{v\in S} (x_v^+ - x_v) \\
& < 3 + \sum_{v\in S} (x_v^+ - x_v).
\end{align*}
It remains to bound $\sum_{v\in S} (x_v^+ - x_v)$. To this end, note that
\begin{equation}\label{eq:sum-x-vs-U}
\sum_{v\in N(v_k)\setminus S} x_v\leq d(v_k) - |S|~~ \text{and} \
\sum_{v\in N(v_k)} x_v = q - (n-k+1) \geq d(v_k) - \frac{1}{2}.
\end{equation}
We immediately obtain 
\[
|S| - \frac{1}{2} \leq\sum_{v\in S} x_v \leq |S|.
\]
Obviously, $\sum_{v\in S} x_v^+ \leq |S|$. Therefore,
\[
\sum_{v\in S} (x_v^+ - x_v) = \sum_{v\in S} x_v^+ - \sum_{v\in S} x_v \leq \frac{1}{2}.
\]
Putting the above inequalities together, we arrive at
\[
x_{k-1}^+ - x_{k-1} < \frac{7}{2(q - |S| - 2)}
< \frac{4}{n-k}.
\]
The last inequality follows from the facts $|S|\leq n-k$ and $q>2(n-k)$.
This completes the proof of the claim.
\end{proof}

To compare $\gamma(G^+)$ with $\gamma(G)$, we use \autoref{lem:gamma} and \autoref{lem:bound-Ui} 
to bound them. On the one hand,
\[
\gamma(G^+) 
= \frac{U_{k-2}^+}{x_{k-1}^+}
> \bigg(q^+ - 2 - \frac{1}{q^+ - 3}\bigg)^{k-3} \cdot\frac{q^+ - 1}{x_{k-1}^+}.
\]
On the other hand, by \autoref{lem:degree-k-2-equal-2} we have
\[
\gamma(G) = \frac{U_{k-2}}{x_{k-1}} < \Big(q - 2 - \frac{1}{q}\Big)^{k-3} \cdot\frac{q - 1}{x_{k-1}}.
\]
Combining \eqref{eq:lambda-plus-lambda} we deduce that
\begin{align*}
\frac{\gamma(G^+)}{\gamma(G)} 
& > \bigg(\frac{q^+ - 2 - (q^+ - 3)^{-1}}{q - 2 - 1/q}\bigg)^{k-3} \cdot\frac{x_{k-1}}{x^+_{k-1}} \\
& > \bigg(1+\frac{3}{2q^2}\bigg)^{k-3} \cdot\frac{x_{k-1}}{x^+_{k-1}}.
\end{align*}
It follows from Bernoulli's inequality that
\begin{equation} \label{eq:ratio-Gplus-G}
\frac{\gamma(G^+)}{\gamma(G)}
> \bigg(1 + \frac{3(k-3)}{2 q^2}\bigg) \cdot\frac{x_{k-1}}{x^+_{k-1}} 
> \bigg(1 + \frac{k}{q^2}\bigg) \cdot\frac{x_{k-1}}{x^+_{k-1}}.
\end{equation}

To finish the proof, we consider the following two cases.

\noindent {\bfseries Case 1.} The maximum eigenvector entry of $\x^+$ is still attained by vertex $v_k$. 

Using the same arguments as \autoref{claim:xu-large}, we get $x_{k-1}^+ > 1/2$. Together with
\autoref{claim:xplus-x} gives
\[
\frac{x_{k-1}}{x_{k-1}^+} > 1 - \frac{4}{(n-k) x_{k-1}^+}
> 1 - \frac{8}{n-k}.
\]
In light of \eqref{eq:ratio-Gplus-G} and \autoref{lem:n-k-value} we find
\[
\frac{\gamma(G^+)}{\gamma(G)} 
> \Big(1 + \frac{k}{q^2}\Big) \Big(1 - \frac{8}{n-k}\Big)
> 1. 
\]
\noindent {\bfseries Case 2.} The maximum eigenvector entry of $\x^+$ is no longer attained 
by vertex $v_k$. 

For any vertex $v\in C$, by the eigenvalue equations for $v_k$ and $v$, we get
\begin{align*}
\big(q^+ - (n-k+1)\big) x_k^+ & = x_{k-1}^+ + \sum_{u\in C} x_u^+, \\
\big(q^+ - d_{G^+}(v)\big) x_v^+ & \leq x_{k-1}^+ + x_k^+ + \sum_{u\in C\setminus\{v\}} x_u^+,
\end{align*}
which imply that $x_k^+ \geq x_v^+$. Hence the maximum entry of $\x^+$ must be 
attained by $v_{k-1}$. Therefore, $\gamma(G^+) = U_{k-2}^+$. Applying \autoref{claim:xplus-x} 
to $x_{k-1}^+=1$, we see $x_{k-1} > 1 - 4/(n-k)$. By \eqref{eq:ratio-Gplus-G} 
again, we have $\gamma(G^+) > \gamma(G)$.

Summing the above two cases, we see $\gamma(G^+)>\gamma(G)$, which is a contradiction to 
the maximality of $\gamma(G)$.
\end{proof}

Based on \autoref{lem:small-q}, we can give a precise estimation for $x_{k-1}$.

\begin{lemma}\label{lem:xk-1-upper-bound}
$x_{k-1} < n^{-1/6}$.
\end{lemma}

\begin{proof}
Let $H=P_{k-1}\cdot K_{n-k+2}$. In view of \autoref{lem:gamma} and \autoref{lem:bound-Ui} we 
conclude that
\begin{align*}
\gamma(H) = U_{k-2}(H)
& \geq \Big(q_1(H) - 2 - \frac{1}{q_1(H) - 3}\Big)^{k-3}\cdot (q_1(H) - 1) \\
& > \Big(2(n-k) - \frac{1}{2(n-k) - 1}\Big)^{k-3} (2n - 2k + 1).
\end{align*}
On the other hand, from \autoref{lem:small-q} we have
\[
\gamma(G) = \frac{U_{k-2}}{x_{k-1}}
< \Big(2(n-k) -\frac{1}{2} - \frac{1}{2(n-k) + 3/2}\Big)^{k-3} \cdot\frac{2(n-k) + 1/2}{x_{k-1}}.
\]
Since $\gamma(G)\geq\gamma(H)$ we deduce that
\begin{align*}
x_{k-1} 
& < \bigg(\frac{2(n-k) - 1/2 - (2n-2k+3/2)^{-1}}{2(n-k) - (2n-2k-1)^{-1}}\bigg)^{k-3} \\
& < \Big(1-\frac{1}{5(n-k)}\Big)^{k-3} \\
& < {\mathbf e}^{-\frac{k-3}{5(n-k)}} \\
& < n^{-1/6}.
\end{align*}
The last inequality uses the fact that $(k-3)/(n-k) > 5\cdot (\log n)/6$ by \autoref{lem:n-k-value}.
\end{proof}

\begin{lemma}\label{lem:degree-vk-1-two}
The degree of $v_{k-1}$ is $2$ in $G$.
\end{lemma}

\begin{proof}
It suffices to show that $|S|=0$. Assume for contradiction that $|S|\geq 1$. Let $G^-$ be the 
graph obtained from $G$ by removing these $|S|$ edges, i.e., $G^- = G - \{v_{k-1}v: v\in S\}$. 
Our goal is to show that $\gamma(G^-)>\gamma(G)$, and therefore get a contradiction. 

We first show that there is at most one vertex in $C$ such that its component in $\x$ no more than 
$(1-x_{k-1})/2$. Indeed, if there are $v,w\in C$ such that $x_v, x_w\leq (1-x_{k-1})/2$, then 
\begin{align*}
n-k-1 
& < \big(q - (n-k+1)\big) x_k \\
& = x_{k-1} + \sum_{u\in C} x_u \\ 
& \leq x_{k-1} + x_v + x_w + (n-k-2) \\ 
& \leq n-k-1,
\end{align*}
which leads to a contradiction. 

We consider the following two cases.
\par\vspace{1.5mm}

\noindent {\bfseries Case 1.} $|S|\geq 2$. By eigenvalue equation we see
\begin{equation}\label{eq:eig-equ-vk-1}
(q - |S| - 2) x_{k-1} > 1 + \sum_{u\in S} x_u > 1 + \frac{(|S| - 1) (1 - x_{k-1})}{2} 
> \frac{|S| (1-x_{k-1})}{2}.
\end{equation}
Solving this inequality gives 
\[
|S| < \frac{2(q - 2) x_{k-1}}{1 + x_{k-1}}.
\]
Combining with \autoref{lem:x-norm} and \autoref{lem:xk-1-upper-bound}, we deduce that
\begin{equation}\label{eq:lambda-lambda-minus}
q - q^- \leq \frac{|S| (1 + x_{k-1})^2}{\|\x\|^2} 
< \frac{2|S| (1 + x_{k-1})^2}{q - 4}
< 5 x_{k-1} < \frac{5}{n^{1/6}}. 
\end{equation}
By \eqref{eq:eig-equ-vk-1} again, we have
\[
(q - |S| - 2) x_{k-1} > 1 + \frac{(|S| - 1) (1 - x_{k-1})}{2},
\]
which implies that 
\begin{equation}\label{eq:xk-1-lower-bound}
x_{k-1} > \frac{|S| + 1}{2q - |S| - 5} \geq \frac{3}{2q - 7}.
\end{equation}
The last inequality is due to our assumption $|S|\geq 2$.

Now we are ready to compare $\gamma(G^-)$ with $\gamma(G)$. By \autoref{lem:bound-Ui} and 
\eqref{eq:xk-1-lower-bound} we see 
\[
\gamma(G)=\frac{U_{k-2}}{x_{k-1}}
< \Big(q - 2 - \frac{1}{q}\Big)^{k-3}\cdot\frac{q - 1}{x_{k-1}}
< \Big(q - 2 - \frac{1}{q}\Big)^{k-3}\cdot \frac{(q - 1)(2q - 7)}{3}.
\]
On the other hand,
\[
\gamma(G^-) = U_{k-1}(G^-) \geq \Big(q^- - 2 - \frac{1}{q^- - 3}\Big)^{k-2} (q^- - 1).
\]
Hence, the above two inequalities, together with \eqref{eq:lambda-lambda-minus}, imply that
\begin{align*}
\frac{\gamma(G^-)}{\gamma(G)} 
& > \frac{6}{5} \cdot \bigg(\frac{q^- - 2 - (q^- - 3)^{-1}}{q - 2 - q^{-1}}\bigg)^{k-3}
> \frac{6}{5} \cdot \bigg(1- \frac{6}{n^{1/6} q}\bigg)^{k-3} >1,
\end{align*}
a contradiction. 
\par\vspace{1.5mm}

\noindent {\bfseries Case 2.} $|S| = 1$. Let $w$ be the unique vertex in $C$ that adjacent 
to $v_{k-1}$. If $x_w > (1 - x_{k-1})/2$, we can deduce a similar contradiction by 
placing \eqref{eq:xk-1-lower-bound} with
\[
x_{k-1} > \frac{1 + x_w}{q - 3} > \frac{6}{5q}
\]
in the proof of Case 1. 

Hence, we assume $x_w \leq (1 - x_{k-1})/2$, and therefore $x_v > (1 - x_{k-1})/2$ for each 
$v\in C\setminus\{w\}$. Evidently, there is a vertex $z\in C\setminus\{w\}$ such that $z$ 
is not adjacent to $w$. Let $\widetilde{G}$ be the graph obtained by deleting edge $v_{k-1} w$ 
and adding $wz$. Since $x_z > x_w$, we obtain $q_1(\widetilde{G}) > q$. Set $\widetilde{q} := q_1(\widetilde{G})$
for short. We have
\[
\gamma(\widetilde{G}) = U_{k-1}(\widetilde{G}) \geq \Big(\widetilde{q} - 2 - \frac{1}{\widetilde{q} - 3} \Big)^{k-2}
(\widetilde{q} - 1) > \Big(q - 2 - \frac{1}{q - 3} \Big)^{k-2} (q - 1).
\]
It follows from $(q - 3) x_{k-1} \geq (q - d(v_{k-1})) x_{k-1} > 1$ that
\begin{align*}
\frac{\gamma (\widetilde{G})}{\gamma (G)}
& > \Big(1 - \frac{4k}{q^3}\Big)\cdot \big(q - 2 - o(1)\big) \cdot x_{k-1} \\
& > \Big(1 - \frac{4k}{q^3}\Big)\cdot \frac{q - 2 - o(1)}{q - 3} \\
& > 1,
\end{align*}
a contradiction completing the proof of Case 2, and hence finishing the proof of 
\autoref{lem:degree-vk-1-two}.
\end{proof}

Now we are ready to prove our main theorem.
\par\vspace{1.5mm}

\noindent {\bfseries Proof of \autoref{thm:Main}.}
According to \autoref{lem:auxiliary-results}, \autoref{lem:degree-k-2-equal-2} and 
\autoref{lem:degree-vk-1-two}, we derive that $v_1,v_2,\ldots,v_k$ form a pendant path.
Hence, it remains to show that the induced subgraph $G[C]$ is a clique. Otherwise, 
let $H = P_{k} \cdot K_{n-k+1}$. Then $G$ is a proper subgraph of $H$, then 
$q_1(G) < q_1(H)$. On the other hand, from \autoref{lem:gamma} and the maximality of 
$\gamma(G)$, we have 
\[
U_{k-1}(G) = \gamma(G) \geq \gamma(H) = U_{k-1}(H),
\]
which yields that $q_1(G)\geq q_1(H)$, a contradiction completing the proof of this theorem.

\end{document}